\documentclass[12pt,reqno]{amsart}

\usepackage{amsfonts,amssymb}
\usepackage{amsmath}
\usepackage{amsthm,amscd,latexsym}
\usepackage{color}

\newtheorem{theorem}{Theorem}[section]
\newtheorem{corollary}[theorem]{Corollary}
\newtheorem{lemma}[theorem]{Lemma}
\newtheorem{proposition}[theorem]{Proposition}

\newtheorem{Definition}[theorem]{Definition}

\newtheorem{Example}[theorem]{Example}
\newtheorem{Remark}[theorem]{Remark}
\newtheorem{problem}[theorem]{Problem}
\numberwithin{equation}{section}

\newenvironment{remark}{\begin{Remark}\begin{em}}{\end{em}\end{Remark}}

\newenvironment{definition}{\begin{Definition}\begin{em}}{\end{em}\end{Definition}}

\newcommand{\Pro}{\mathcal P}

\def\cQ{\mathcal{Q}}
\def\bP{\mathbb{P}}
\def\cP{\mathcal{P}}

\def\bH{\mathbb{H}}

\def\<{\langle}
\def\>{\rangle}

\def\bN{\mathbb{N}}

\DeclareMathOperator{\ua}{\uparrow\!}

\DeclareMathOperator{\tr}{tr}
\DeclareMathOperator{\argmin}{\mathrm{arg\, min}}
\begin{document}
\allowdisplaybreaks

\title[Log-majorization and Lie-Trotter formula]{Log-majorization and
Lie-Trotter formula for the Cartan barycenter on probability measure spaces}
\author[Hiai and Lim]{Fumio Hiai and Yongdo Lim}
\address{Tohoku University (Emeritus), Hakusan 3-8-16-303, Abiko 270-1154, Japan}\email{hiai.fumio@gmail.com}
\address{Department of Mathematics, Sungkyunkwan University, Suwon 440-746, Korea} \email{ylim@skku.edu}
\date{\today}
\maketitle

\begin{abstract}

We extend Ando-Hiai's log-majorization for the weighted geometric
mean of positive definite matrices into that for the Cartan
barycenter in the general setting of probability measures on the
Riemannian manifold of positive definite matrices equipped with
trace metric. The main key is the settlement of the
monotonicity problem of the Cartan barycenteric map on the space of
probability measures with finite first moment for the stochastic
order induced by the cone. We also derive a version of Lie-Trotter
formula and related unitarily invariant norm inequalities for the
Cartan barycenter as the main application of log-majorization.
\end{abstract}

\medskip
\noindent \textit{2010 Mathematics Subject Classification}. 15A42,
47A64, 47B65, 47L07

\noindent \textit{Key words and phrases.} Positive definite matrix,
Cartan barycenter, Wasserstein distance, log-majorization,
Lie-Trotter formula, unitarily invariant norm

\section{Introduction}
Let $A$ be an $m\times m$ positive definite matrix  with eigenvalues
$\lambda_{j}(A)$, $1\le j\le m$, arranged in decreasing order, i.e.,
$\lambda_{1}(A)\geq \cdots\geq \lambda_{m}(A)$ with counting multiplicities.
The {\it log-majorization} $A\underset{\log}{\prec} B$ between  positive definite
matrices $A$ and $B$ is defined if
$$\prod_{i=1}^{k}\lambda_{i}(A)  \leq \prod_{i=1}^{k}\lambda_{i}(B) \quad\mbox{for }
1\le k\le m-1,\mbox{ and } \det A  = \det B.
$$
The log-majorization gives  rise to powerful devices in deriving
various norm inequalities and has many important applications in
operator means, operator monotone functions, statisticalmechanics,
quantum information theory, eigenvalue analysis, etc., see, e.g.,
{\cite{BLP,BG,HiP1}}. For instance,
$A\underset{\log}{\prec} B$ implies $||| A|||\leq |||B|||$
for all unitarily invariant norms $|||\cdot |||$.

As a complementary counterpart of the Golden-Thompson trace inequality,
Ando and Hiai \cite{AH} established the log-majorization on the matrix
geometric mean of two positive definite matrices: for positive
definite matrices $A,B$ and $0\le\alpha\leq 1$,
\begin{eqnarray*}
A^{t}\#_{\alpha}B^{t}\underset{\log}{\prec} (A\#_{\alpha} B)^{t},
\qquad t\geq 1,
\end{eqnarray*}
where $A\#_{\alpha}B:=A^{1/2}(A^{-1/2}BA^{-1/2})^{\alpha}A^{1/2}$,
the {\it $\alpha$-weighted geometric mean} of $A$ and $B$.
This provides various norm inequalities for unitarily invariant norms via
the Lie-Trotter formula $
\lim_{t\to0}(A^{t}\#_{\alpha}B^{t})^{\frac{1}{t}}=e^{(1-\alpha)\log
A+\alpha \log B}$. For instance,
$|||(A^{t}\#_{\alpha}B^{t})^{\frac{1}{t}}|||$ increases to
$|||e^{(1-\alpha)\log A+\alpha \log B}|||$ as $r\searrow 0$ for any
unitarily invariant norm. Ando-Hiai's log-majorization has many
important applications in matrix analysis and inequalities, together
with Araki's log-majorization \cite{Ar} extending the Lieb-Thirring
and the Golden-Thompson trace inequalities.

The matrix geometric mean $A\#_{\alpha}B$, that plays the central
role in Ando-Hiai's log-majorization, appears as the unique (up to
parametrization) geodesic curve $\alpha\in[0,1]\mapsto
A\#_{\alpha}B$ between $A$ and $B$ on the Riemannian manifold ${\Bbb
P}_{m}$ of positive definite matrices of size $m$, an important
example of Cartan-Hadamard Riemannian manifolds. Alternatively, the
geometric mean $A\#_{\alpha}B$ is the Cartan barycenter of the
finitely supported measure $(1-\alpha)\delta_{A}+\alpha \delta_{B}$
on ${\Bbb P}_{m}$, which is defined as the unique minimizer of the
least squares problem with respect to the Riemannian distance $d$
(see Section 2 for definition). Indeed, for a general probability
measure $\mu$ on $\bP_m$ with finite first moment, the Cartan
barycenter of $\mu$ is defined as the unique minimizer as follows:
\begin{displaymath}
G(\mu) := \underset{Z \in \mathbb{P}_m}{\argmin} \int_{\mathbb{P}_m}
\bigl[d^{2} (Z, X)-d^2(Y,X)\bigr]d\mu(X)
\end{displaymath}
(see Section 2 for more details).
In particular, when $\mu=\sum_{j=1}^nw_j\delta_{A_j}$ is a discrete probability measure
supported on a finite number of $A_1,\dots,A_n\in\bP_m$, the Cartan
barycenter $G(\mu)$ is the {\it Karcher mean} of $A_1,\dots,A_n$, which has extensively
been discussed in these years by many authors as a multivariable extension of the
geometric mean (see \cite{BH,LL1,Ya1} and references therein).

The first aim of this paper is to establish the log-majorization
(Theorem \ref{T:MAIN}) for the Cartan barycenter in the general
setting of probability measures in the Wasserstein space ${\mathcal
P}^{1}({\Bbb P}_{m})$, the probability measures on $\bP_m$ with
finite first moment. In this way, we first establish the monotonicity
of the Cartan barycenteric map on ${\mathcal P}^{1}({\Bbb P}_{m})$ for
the stochastic order induced by the cone of positive semidefinite matrices,
and then generalize the log-majorization in \cite{AH} (as mentioned above)
and in \cite{HiP} (for the Karcher mean of multivariables) to the setting of
probability measures. Our second aim is to derive the Lie-Trotter
formula (Theorem \ref{T8}) for the Cartan barycenter
\begin{equation*}
\lim_{t\to0}G(\mu^t)^{1\over t}=\exp\int_{\bP_m}\log A\,d\mu(A)
\end{equation*}
under a certain integrability assumption on $\mu$, where $\mu^{t}$
is the $t$th  power of the measure $\mu$ inherited from the matrix
powers on ${\Bbb P}_{m}$. Moreover, to demonstrate the usefulness of
our log-majorization, we obtain several unitarily invariant norm
inequalities (Corollary \ref{C9}) based on the above Lie-Trotter
formula.

The main tools of the paper involve the theory of nonpositively
curved metric spaces and techniques from probability measures on
metric spaces  and the recent combination of the two (see
\cite{St03,AGS,Vi1}).  Not only are these tools crucial for our
developments, but also, we believe, significantly enhance the
potential usefulness of the Cartan barycenter of probability
measures in matrix analysis and inequalities. They overcome the
limitation to the multivariable (finite number of matrices) setting,
and provide a new bridge between two different important fields of
studies of matrix analysis and probability measure theory on
nonpositively curved metric spaces.

\section{Cartan barycenters}

Let ${\Bbb H}_{m}$ be the Euclidean space of $m \times m$ Hermitian
matrices equipped with the inner product $\langle X,Y \rangle :=
{\mathrm{tr}}(XY)$. The {\it Frobenius norm} $\|\cdot\|_{2}$ defined
by $\|X\|_{2} = (\tr X^{2})^{1/2}$ for $X \in {\Bbb H}_m$ gives rise
to the Riemannian structure on the open convex cone ${\Bbb P}_{m}$
of $m\times m$ positive definite matrices with the metric
\begin{equation}\label{metric}
\langle X,Y\rangle_{A} := {\mathrm{tr}}(A^{-1} X A^{-1} Y), \qquad
A\in\bP_m,\ X,Y\in\bH_m,
\end{equation}
where the tangent space of $\bP_m$ at any point $A\in\bP_m$ is
identified with $\bH_m$. The Riemannian exponential
at $A \in {\Bbb P}_{m}$ is given by
\begin{eqnarray*}\exp_{A}(X) =
A^{\frac{1}{2}}\exp(A^{-\frac{1}{2}}XA^{-\frac{1}{2}})A^{\frac{1}{2}}\end{eqnarray*}
and its inverse is
\begin{eqnarray*}\log_{A}(X) =
A^{\frac{1}{2}}\log(A^{-\frac{1}{2}}XA^{-\frac{1}{2}})A^{\frac{1}{2}}.\end{eqnarray*}
Then ${\Bbb P}_m$ is a {\it Cartan-Hadamard Riemannian manifold}, a
simply connected complete Riemannian manifold with nonpositive
sectional curvature (the canonical $2$-tensor is nonnegative). The
{\it Riemannian trace metric} (i.e., the geodesic distance with
respect to \eqref{metric}) on ${\Bbb P}_m$ is given by
$$
d(A,B) := \big\| \log A^{-\frac{1}{2}} B A^{-\frac{1}{2}} \big\|_{2},
$$
and the unique (up to parametrization) geodesic shortest curve
joining $A$ and $B$ is $t\in[0,1] \mapsto A \#_{t} B =
A^{\frac{1}{2}}(A^{-\frac{1}{2}}BA^{-\frac{1}{2}})^{t}A^{\frac{1}{2}}$.
The nonpositively curved property is equivalently stated as
\begin{eqnarray}\label{NP}d^{2}(A\#_{t}B,C)&\leq&
(1-t)d^{2}(A,C)+td^{2}(B,C)-(1-t)td^{2}(A,B).
\end{eqnarray}
See \cite{LL01,Bh} for more about these Riemannian structures.

Let $\mathcal{B}={\mathcal B}({\Bbb P}_m)$ be the algebra of Borel
sets, the smallest $\sigma$-algebra containing the open sets of ${\Bbb P}_m$. We note that
the Euclidean topology on ${\Bbb P}_m$ coincides with the metric topology of the trace metric
$d$. Let ${\mathcal P}={\mathcal P}(\bP_m)$ be the set of all probability measures on
$({\Bbb P}_m, {\mathcal B})$ and ${\mathcal P}_c={\mathcal P}_c(\bP_m)$ the set of
all compactly supported $\mu\in{\mathcal P}$. Let ${\mathcal P}_{0}={\mathcal P}_{0}(\bP_m)$
be the set of all $\mu \in {\mathcal P}$ of the form
$\mu = (1/n) \sum_{j=1}^{n} \delta_{A_{j}}$, where $\delta_A$ is the point measure of mass
$1$ at $A \in {\Bbb P}$. For $p\in[1,\infty)$ let ${\mathcal P}^{p}={\mathcal P}^{p}(\bP_m)$
be the set of probability measures with \emph{finite $p$-moment}, i.e., for some (and hence
all) $Y\in {\Bbb P}_m$,
$$ \int_{ {\Bbb P}_m} d^p(X,Y)\,d\mu(X) < \infty. $$

We say that $\omega \in \Pro({\Bbb P}_m
\times {\Bbb P}_m)$ is a \emph{coupling} for $\mu,\nu \in \Pro$ if
 $\mu,\nu$ are the marginals of $\omega$, i.e.,  if for all $B \in\mathcal{B}$,
$ \omega(B \times {\Bbb P}_m) = \mu(B)$ and  $\omega({\Bbb P}_m \times
B) = \nu(B)$. We note that one such coupling is the product measure
$\mu \times \nu$.  We denote the set of all couplings for $\mu,\nu
\in \Pro(\bP_m)$ by $\Pi(\mu,\nu)$.

The $p$-\emph{Wasserstein distance} $d_{p}^W$ on ${\mathcal P}^{p}$ is defined by
$$ d_{p}^{W}(\mu,\nu) := \left[ \inf_{\pi \in \Pi(\mu,\nu)}
\int_{{\Bbb P}_m \times {\Bbb P}_m} d^p(X,Y)\,d\pi(X,Y) \right]^{\frac{1}{p}}. $$
It is known that $d_{p}^W$ is a complete metric on ${\mathcal P}^{p}$ and
${\mathcal P}_{0}$ is dense in ${\mathcal P}^{p}$ \cite{St03}.
Note that $\mathcal{P}_{0} \subset \mathcal{P}_c \subset \mathcal{P}^{q} \subset
\mathcal{P}^{p} \subset \mathcal{P}^{1}$ and  $d^{W}_{p} \leq d^{W}_{q}$ for
$1 \leq p \leq q < \infty$.
We note that these basic results on probability measure spaces hold
in general setting of complete metric spaces in which cases
separability assumption is necessary.

The following result on Lipschitz property of push-forward maps between metric spaces
appears in \cite{LL5}, where $X,Y$ are metric spaces and the
distance $d_p^W$ on $\cP^p(X),\cP^p(Y)$ are defined as above.

\begin{lemma}\label{L:lip}
Let $f:X\to Y$ be a Lipschitz map with Lipschitz constant $C$.  Then
the push-forward map  $f_*:\mathcal{P}^p(X)\to \mathcal{P}^p(Y)$,
$f_{*}(\mu)=\mu\circ f^{-1}$, is Lipschitz with respect to $d_p^W$
with Lipschitz constant $C$ for $1\leq p<\infty$.
\end{lemma}

\begin{definition}
The {\it Cartan barycenter} map $G: {\mathcal P}^1({\Bbb P}_m)\to {\Bbb P}_m$ is
defined by
\begin{displaymath}
G(\mu) := \underset{Z \in \mathbb{P}_m}{\argmin} \int_{\mathbb{P}_m}
\bigl[d^{2} (Z, X)-d^2(Y,X)\bigr]\,d\mu(X),\qquad\mu\in\cP^1(\bP_m)
\end{displaymath}
for a fixed $Y$. The uniqueness and existence of the minimizer is well-known and the
unique minimizer is independent of $Y$ (see \cite[Proposition 4.3]{St03}). On
${\mathcal P}^2({\Bbb P}_m)$, the Cartan barycenter is determined by
\begin{displaymath}
G(\mu) = \underset{Z \in \mathbb{P}_m}{\argmin} \int_{\mathbb{P}_m}
d^{2} (Z, X)\,d\mu(X).
\end{displaymath}
For a discrete measure $\mu=\sum_{j=1}^{n}w_j\delta_{A_{j}}$, $G(\mu)$ is the
Karcher mean of $A_1,\dots,A_n$ with a weight $(w_1,\dots,w_n)$, see, e.g., \cite{LL1,Ya1}.
\end{definition}

The following contraction property appears in \cite{St03}.

\begin{theorem}[Fundamental Contraction Property] \label{T:ft}
For every $\mu,\nu \in {\mathcal P}^p(\bP_m)$, $p\ge1$,
$$ d(G(\mu),G(\nu)) \leq d_{1}^{W}(\mu,\nu)\leq d_p^W(\mu,\nu). $$
\end{theorem}

\section{Karcher equations and monotonicity}
A map $g:{\Bbb P}_{m}\to {\Bbb R}$ is called \emph{uniformly convex}
if there is a strictly increasing function $\phi:[0,\infty)\to
[0,\infty)$ such that
$$g(A\#B)\leq \frac{1}{2}(g(A)+g(B))-\phi(\delta(A,B))$$ for all
$A,B\in {\Bbb P}_{m}$. For a continuous uniformly convex function
$g$, it has a unique minimizer of $g$ (see \cite{St03}) and
coincides with the unique point that vanishes the (either Riemannian
or Euclidean) gradient, whenever it is differentiable, see
\cite{LL13}.

By (\ref{NP}), the map
$$Z\mapsto \int_{\mathbb{P}_m}\bigl[d^{2} (Z,
X)-d^2(Y,X)\bigr]\,d\mu(X),\qquad\mu\in\cP^1(\bP_m)$$ is uniformly
convex. The next theorem is a characterization of $G(\mu)$ in terms
of the unique solution to the Karcher equation.

\begin{theorem}\label{T:kare}
For every $\mu\in\cP^1(\bP_m)$, $G(\mu)$ is the unique solution
$Z\in\bP_m$ to the Karcher equation
$$
\int_{\bP_m}\log Z^{-1/2}XZ^{-1/2}\,d\mu(X)=0.
$$
\end{theorem}
\begin{proof} Let $\mu\in\cP^1(\bP_m)$.
We first show that the Euclidean gradient of the function
$$
Z\in\bP_m\ \longmapsto\
\varphi(Z):=\int_{\bP_m}\bigl[d^2(Z,X)-d^2(Y,X)\bigr]\,d\mu(X)
$$
is
$$
Z^{-1/2}\biggl(\int_{\bP_m}\log Z^{1/2}X^{-1}Z^{1/2}\,d\mu(X)\biggr)Z^{-1/2}.
$$
More precisely, with
$$
F(Z,X):=2Z^{-1/2}\bigl(\log
Z^{1/2}X^{-1}Z^{1/2}\bigr)Z^{-1/2},\qquad Z,X\in\bP_m,
$$
we shall prove that
\begin{equation}\label{F-3.1}
\varphi(Z+H)=\varphi(Z)+\int_{\bP_m}\tr F(Z,X)H\,d\mu(X)+o(\|H\|_2)
\end{equation}
as $\|H\|_2\to0$ for $H\in\bH_m$.

For each fixed $X\in\bP_m$, let
$$
\psi(Z):=d^2(Z,X)=\tr\bigl(\log X^{-1/2}ZX^{-1/2}\bigr)^2,\qquad
Z\in\bP_m.
$$
It is not difficult to compute the gradient of $\psi(Z)$ is
$F(Z,X)$, i.e.,
$$
\psi(Z+H)=\psi(Z)+\tr F(Z,X)H+o(\|H\|_2)
$$
as $\|H\|_2\to0$ for $H\in\bH_m$. Then, for every $Z\in\bP_m$ and
$H\in\bH_m$, by the Lebesgue convergence theorem one can prove that
\begin{align*}
{d\over dt}\,\varphi(Z+tH)\Big|_{t=0}
&=\lim_{t\to0}{\varphi(Z+tH)-\varphi(Z)\over t} \\
&=\lim_{t\to0}\int_{\bP_m}{\psi(Z+tH)-\psi(Z,X)\over t}\,d\mu(X) \\
&=\int_{\bP_m}\tr F(Z,X)H\,d\mu(X).
\end{align*}
This formula for the directional derivative is enough to give
\eqref{F-3.1} (due to the finite dimensionality).
\end{proof}

For any ${\mathcal U}\subset {\Bbb P}_{m}$, we define ${\mathcal
U}^{\ua} = \{B\in {\Bbb P}_{m} : A \leq B \ {\mathrm{for\ some}}\ A
\in {\mathcal U}\}$. A set ${\mathcal U}$ is an \emph{upper set} if
${\mathcal U}^{\ua}={\mathcal U}$. For $\mu,\nu\in \mathcal{P}({\Bbb
P}_{m})$, we define $\mu\leq \nu$ if $\mu({\mathcal U})\leq
\nu({\mathcal U})$ for all open upper sets $\mathcal U$. This
partial order on $\mathcal{P}({\Bbb P}_{m})$ is a natural extension
of the usual one; $A_{j}\leq B_{\sigma(j)}$ for some permutation
$\sigma$ and $j=1,\dots, n$ if and only if
$(1/n)\sum_{j=1}^{n}\delta_{A_{j}}\leq
(1/n)\sum_{j=1}^{n}\delta_{B_{j}}$, as seen from the marriage
theorem.

We recall the well-known  L\"owner-Heinz
inequality:
$$0<A\leq B \ \ \  {\mathrm{implies}} \ \ \ A^{t}\leq B^{t}, \ \ t\in [0,1].$$
The next theorem is the monotonicity property of the Cartan
barycenter $G$ on $\cP^1(\bP_m)$ and extends the recent works of
Lawson-Lim \cite{LL1} and Bhatia-Karandikar \cite{BK} on the space
of finitely (and uniformly) supported measures, which can be viewed
as a multivariate L\"owner-Heinz inequality.

\begin{theorem}\label{T:mono}
Let $\mu,\nu\in\cP^1(\bP_m)$. If $\mu\le\nu$, then $G(\mu)\le
G(\nu)$.
\end{theorem}

\begin{proof}
Assume that $\mu,\nu\in\cP^1(\bP_m)$ and $\mu\le\nu$. For each
$n\in\bN$ let $\Sigma_n:=\{X\in\bP_m:(1/n)I\le X\le nI\}$ and
$$
\mu_n:=\mu|_{\Sigma_n}+\mu(\bP_m\setminus\Sigma_n)\delta_{(1/n)I},\qquad
\nu_n:=\nu|_{\Sigma_n}+\nu(\bP_m\setminus\Sigma_n)\delta_{nI}.
$$
Then, as in the proof of \cite[Section 6]{KL}, we have
$\mu_n\le\nu_n$. Since $\mu_n,\nu_n\in\cP_c(\bP_m)$, we have
$G(\mu_n)\le G(\nu_n)$ by \cite[Theorem 5.5\,(6)]{KL}. We now prove that
$d_1^W(\mu_n,\mu)\to0$ as $n\to\infty$. From a basic fact on the
convergence in Wasserstein spaces (see \cite[Theorem 7.12]{Vi1}) we may prove
that $\mu_n\to\mu$ weakly and
\begin{equation*}
\lim_{n\to\infty}\int_{\bP_m}\|\log
X\|_2\,d\mu_n(X)=\int_{\bP_m}\|\log X\|_2\,d\mu(X).
\end{equation*}
Since $\mu(\bP_m\setminus\Sigma_n)\to0$, it is obvious that
$\mu_n\to\mu$ weakly. Note that
\begin{align*}
\int_{\bP_m}\|\log X\|_2\,d\mu_n(X)
&=\int_{\Sigma_n}\|\log X\|_2\,d\mu(X)+\|\log((1/n)I)\|_2\,\mu(\bP_m\setminus\Sigma_n) \\
&=\int_{\Sigma_n}\|\log X\|_2\,d\mu(X)+\sqrt m\,(\log
n)\,\mu(\bP_m\setminus\Sigma_n).
\end{align*}
Note also that if $X\in\bP_m\setminus\Sigma_n$, then either the largest
eigenvalue of $X$ satisfies $\lambda_1(X)>n$ or the smallest one does $\lambda_m(X)<1/n$,
so we have $\|\log X\|_2\ge\log n$. Therefore, since
$\int_{\bP_m}\|\log X\|_2\,d\mu(X)<\infty$, we have
$$
(\log
n)\,\mu(\bP_m\setminus\Sigma_n)\le\int_{\bP_m\setminus\Sigma_n}\|\log
X\|_2\,d\mu(X) \ \longrightarrow\ 0\quad\mbox{as $n\to\infty$},
$$
so that
$$
\lim_{n\to\infty}\int_{\bP_m}\|\log X\|_2\,d\mu(X)
=\lim_{n\to\infty}\int_{\Sigma_n}\|\log
X\|_2\,d\mu(X)=\int_{\bP_m}\|\log X\|_2\,d\mu(X).
$$
We thus have $d_1^W(\mu_n,\mu)\to0$, which implies
$\delta(G(\mu_n),G(\mu))\to0$ by the fundamental contraction
property, so $\|G(\mu_n)-G(\mu)\|_2\to0$. Since
$\|G(\nu_n)-G(\nu)\|_2\to0$ similarly, $G(\mu)\le G(\nu)$ follows by
taking the limit of $G(\mu_n)\le G(\nu_n)$.
\end{proof}

\section{Log-majorization}

For $1\leq k\leq m$ and $A\in {\Bbb P}_{m}$, let $\Lambda^{k}A$ be
the $k$th {\it antisymmetric tensor power} of $A$. See \cite{AH,BK,HiP1}
for basic properties of $\Lambda^{k}$; for instance,
\begin{eqnarray}
\Lambda^{k}(AB)&=& (\Lambda^{k}A)(\Lambda^{k}B), \nonumber\\
\label{step 3}\Lambda^{k}(A^{t})&=&(\Lambda^{k}A)^{t}, \ \ t>0,\\
\label{step2}\lambda_{1}(\Lambda^{k}A)&=&\prod_{j=1}^{k}\lambda_{j}(A).
\end{eqnarray}

The $k$th antisymmetric tensor power map $\Lambda^k$ maps ${\Bbb P}_{m}$
continuously into ${\Bbb P}_{\ell}$ where $\ell:={m\choose k}$.
This induces the push-forward map
$$\Lambda^{k}_{*}: {\mathcal P}({\Bbb P}_{m})\to {\mathcal P}({\Bbb
P}_{\ell}), \qquad
\Lambda^{k}_{*}(\mu):=\mu\circ (\Lambda^{k})^{-1},$$
that is, $\Lambda^{k}_{*}(\mu)({\mathcal
O})=\mu((\Lambda^{k})^{-1}({\mathcal O}))$ for all Borel sets
${\mathcal O}\subset{\Bbb P}_{\ell}$.

\begin{proposition}\label{P:Lip}
The map $\Lambda^{k}:{\Bbb P}_{m}\to {\Bbb P}_{\ell}$ is
Lipschitzian, that is,
$$d(\Lambda^{k}A,\Lambda^{k}B)\leq \alpha_{m,k}\ d(A,B),\qquad
A,B\in {\Bbb P}_{m},
$$
where
$\alpha_{m,k}:=\sqrt{k{m-1\choose k-1}}$.
Furthermore, $\Lambda^{k}_{*}: {\mathcal P}^p({\Bbb P}_{m})\to
{\mathcal P}^p({\Bbb P}_{\ell})$ is Lipschitzian for every $p\geq 1$, that is,
$$d_{p}^W(\Lambda^k_{*}(\mu),\Lambda^k_{*}(\nu))\leq
\alpha_{m,k} \ d_{p}^W(\mu,\nu),  \ \ \  \ \mu,\nu\in {\mathcal
P}^{p}({\Bbb P}_{m}).$$
\end{proposition}

\begin{proof}
The eigenvalue list of
$\Lambda^{k}(A^{-\frac{1}{2}}BA^{-\frac{1}{2}})=
(\Lambda^{k}A)^{-\frac{1}{2}}(\Lambda^{k}B)(\Lambda^{k}A)^{-\frac{1}{2}}$
is
$$\prod_{j=1}^{k}\lambda_{i_{j}}(A^{-\frac{1}{2}}BA^{-\frac{1}{2}}),\qquad
1\leq i_{1}<\cdots<i_{k}\leq m.$$ Hence
\begin{eqnarray*}
d^{2}(\Lambda^kA,\Lambda^kB)&=&\big\|\log
\Lambda^k(A^{-\frac{1}{2}}BA^{-\frac{1}{2}})\big\|_{2}^2\\&=&
\sum_{1\leq i_{1}<\cdots<i_{k}\leq m}\log^2
\left(\prod_{j=1}^{k}\lambda_{i_{j}}(A^{-\frac{1}{2}}BA^{-\frac{1}{2}})\right)
\\&=&
\sum_{1\leq i_{1}<\cdots<i_{k}\leq m}\left[\sum_{j=1}^k\log
\lambda_{i_{j}}(A^{-\frac{1}{2}}BA^{-\frac{1}{2}})\right]^2
\\&\le&
\sum_{1\leq i_{1}<\cdots<i_{k}\leq
m}k\sum_{j=1}^k\log^2
\lambda_{i_{j}}(A^{-\frac{1}{2}}BA^{-\frac{1}{2}})\\
&=&k{m-1\choose k-1}\sum_{i=1}^m\log^2
\lambda_{i}(A^{-\frac{1}{2}}BA^{-\frac{1}{2}}) \\
&=&k{m-1\choose k-1}d^{2}(A,B).
\end{eqnarray*}
The Lipschitz continuity of $\Lambda^{k}_{*}$ follows by Lemma
\ref{L:lip}.
\end{proof}

The following is an extension of the result by  Bhatia and Karandikar \cite[Theorem 4.4]{BK}
for finitely supported measures to general probability measures in $\cP^1(\bP_m)$.

\begin{theorem} For $p\geq 1$,  the following diagram commute:
\[ \begin{CD}
{\Bbb P}_{m}  @>\Lambda^{k}>> {\Bbb P}_{\ell} \\
@AGAA     @AAGA\\
{\mathcal P}^p({\Bbb P}_{m}) @>\Lambda^{k}_{*}>> {\mathcal
P}^p({\Bbb P}_{\ell})
\end{CD},
\]
that is,
\begin{eqnarray}\label{E:comu}G\circ \Lambda^{k}_{*}=\Lambda^{k}\circ G.
\end{eqnarray}
\end{theorem}

\begin{proof}
Let $\mu\in {\mathcal P}^{1}({\Bbb P}_{m})$.
By Theorem \ref{T:kare}, letting $Z:=G(\mu)$, we may prove that
$$
\int_{\bP_{\ell}}\log
\bigl(\Lambda^k(Z)\bigr)^{-1/2}X\bigl(\Lambda^k(Z)\bigr)^{-1/2}\,d(\Lambda^k_*\mu)(X)=0,
$$
i.e., $
\int_{\bP_m}\log\bigl[\Lambda^k\bigl(Z^{-1/2}XZ^{-1/2}\bigr)\bigr]\,d\mu(X)=0$. Note that
\begin{eqnarray*}
\log\bigl[\Lambda^k\bigl(Z^{-1/2}XZ^{-1/2}\bigr)\bigr]
&=&\log\bigl(Z^{-1/2}XZ^{-1/2}\bigr)^{\otimes k}\Big|_{({\Bbb C}^m)^{\Lambda k}} \\
&=&\Biggl(\sum_{j=1}^kI^{\otimes(j-1)}\otimes\bigl(\log
Z^{-1/2}XZ^{-1/2}\bigr) \otimes
I^{\otimes(k-j)}\Biggr)\Bigg|_{({\Bbb C}^m)^{\Lambda k}},
\end{eqnarray*}
where $({\Bbb C}^m)^{\Lambda k}$ is the $k$-fold antisymmetric tensor
space of ${\Bbb C}^m$.

Since $\int_{\bP_m}\log Z^{-1/2}XZ^{-1/2}\,d\mu(X)=0$, we have
\begin{align*}
&\int_{\bP_m}\log\bigl[\Lambda^k\bigl(Z^{-1/2}XZ^{-1/2}\bigr)\bigr]\,d\mu(X) \\
&\qquad=\Biggl(\sum_{j=1}^kI^{\otimes(j-1)}\otimes
\biggl(\int_{\bP_m}\log Z^{-1/2}XZ^{-1/2}\,\mu(X)\biggr) \otimes
I^{\otimes(k-j)}\Biggr)\Bigg|_{({\Bbb C}^m)^{\Lambda k}}=0.
\end{align*}
\end{proof}

Next, we introduce powers of probability measures on ${\Bbb P}_m$.
\begin{definition}
For $t\in {\Bbb R}\setminus \{0\}$ and
$\mathcal{O}\in {\mathcal B}({\Bbb P}_m)$, we let $ \mathcal{O}^{t}
:= \{ A^{t}: A \in \mathcal{O} \} $ and
\begin{eqnarray*}
\mu^{t} (\mathcal{O}) := \mu(\mathcal{O}^{\frac{1}{t}}).
\end{eqnarray*}
In terms of push-forward measures, $\mu^{t}=g_{*}\mu$,  where $g(X):=X^{t}$. Note that
$\mu^{t}\in {\mathcal P}^p(\bP_m)$ if $\mu\in {\mathcal P}^p(\bP_m)$.
\end{definition}

By (\ref{step 3}) and the definition of push-forward map, we have
\begin{eqnarray}\label{E:Hiai}
\Lambda^{k}_{*}(\mu^{t})=\Lambda^{k}_{*}(\mu)^t, \qquad
\mu\in {\mathcal P}^p(\bP_m), \ t\neq 0.
\end{eqnarray}

 In \cite{KLL},
Kim-Lee-Lim established that $\|G(\mu^{t})\|\leq \|G(\mu)^{t}\| $
for $\mu\in {\mathcal P}^2(\bP_m)$ and $t\geq 1$, where $\|\cdot\|$
denotes the operator norm. It follows from the monotonicity of
Cartan barycenter and its the characterization via the Karcher equation.
In the present situation, the same method based on Theorems \ref{T:kare}
and \ref{T:mono} proves that
\begin{eqnarray}\label{T:AHI}\|G(\mu^{t})\|\leq \|G(\mu)^{t}\|, \qquad
\mu\in {\mathcal P}^1(\bP_m), \ \ t\geq 1.
\end{eqnarray}

The main result of this section is the following:

\begin{theorem}\label{T:MAIN}
For every $\mu\in {\mathcal P}^1(\bP_m)$ and $t\geq 1$,
\begin{eqnarray*}
G(\mu^{t})\underset{\log}{\prec}G(\mu)^{t}.
\end{eqnarray*}
In particular, for any unitary invariant norm $|||\cdot|||$,
$$|||G(\mu^{t})|||\leq |||G(\mu)^{t}|||,\qquad t\ge1.$$
\end{theorem}
\begin{proof}
For $1\le k\le m$ we have
\begin{eqnarray*}
\prod_{j=1}^{k}\lambda_{j}(G(\mu^{t}))&=&\lambda_{1}(\Lambda^{k}G(\mu^{t}))
=\|\Lambda^{k}G(\mu^{t})\|\\
&=&\big\|G\bigl(\Lambda^{k}_{*}(\mu^{t})\bigr)\big\|
=\big\|G\bigl((\Lambda^{k}_{*}(\mu))^{t}\bigr)\big\| \\
&\leq&\big\|G\bigl((\Lambda^{k}_{*}(\mu))\bigr)\big\|^t
=\prod_{j=1}^{k}\lambda_{j}(G(\mu)^{t}),
\end{eqnarray*}
where \eqref{step2}, \eqref{E:comu}, \eqref{E:Hiai}, and
\eqref{T:AHI} have been used. It remains to show that $\det G(\mu^t)=\det
G(\mu)^t$. When $k=m$, since $\Lambda^m=\det$ and
$G\bigl((\Lambda^m_*(\mu))^t\bigr)$ is a positive scalar, the equalities shown
above say that $\det G(\mu^t)=G\bigl((\Lambda^m_*(\mu))^t\bigr)$. In the
one-dimensional case on $\bP_1=(0,\infty)$, we find by a direct
computation that
$$
G(\nu)=\exp\int_{(0,\infty)}\log x\,d\nu(x)
$$
for every $\nu\in\cP^1((0,\infty))$. Therefore,
\begin{align*}
G\bigl((\Lambda^m_*(\mu))^t\bigr)&=\exp\int_{(0,\infty)}\log
x\,d(\Lambda^m_*(\mu)^t)(x)
=\exp\int_{\bP_m}\log({\det}^tA)\,d\mu(A) \\
&=\exp\int_{\bP_m}t\tr(\log A)\,d\mu(A)
={\det}^t\biggl(\exp\int_{\bP_m}\log A\,d\mu(A)\biggr) \\
&=\det G(\mu)^t,
\end{align*}
implying that $\det G(\mu^t)=\det G(\mu)^t$.
\end{proof}

By a consequence of the preceding theorem, we have the following:
\begin{corollary} For every $\mu\in {\mathcal P}^1({\Bbb P}_{m})$,
\begin{equation*}
G(\mu^{q})^{\frac{1}{q}}\underset{(\log)}{\prec}G(\mu^{p})^{\frac{1}{p}},
\qquad 0<p\leq q,
\end{equation*}
\begin{equation*}
G(\mu^{p})^{\frac{1}{p}}\underset{(\log)}{\prec}G(\mu)\underset{(\log)}{\prec}
G(\mu^{\frac{1}{p}})^{p},\qquad p\geq 1,
\end{equation*}
and therefore
\begin{equation}\label{E:LL}
|||G(\mu^{q})^{\frac{1}{q}}|||\leq
|||G(\mu^{p})^{\frac{1}{p}}|||,
\qquad 0<p\leq q,
\end{equation}
\begin{equation*}
|||G(\mu^{p})^{\frac{1}{p}}|||\leq |||G(\mu)|||\leq
|||G(\mu^{\frac{1}{p}})^{p}|||, \ \ \ \ p\geq 1
\end{equation*}
for all unitarily invariant norms $|||\cdot |||$.
\end{corollary}

\section{Lie-Trotter formula}

The Lie-Trotter formula for the Cartan (or Karcher) mean of multivariable
positive definite matrices is
\begin{equation*}
{\underset{t\to 0}{\lim}}\, G(
A_1^t,\ldots,A_n^t)^{\frac{1}{t}}=\exp\left(\frac{1}{n}
\sum_{j=1}^{n}\log A_{j}\right),
\end{equation*}
see \cite{HiP,BJL,BG}.

In this section we establish the Lie-Trotter formula and associated norm
inequalities for probability measures in a certain sub-class of $\cP^1(\bP_m)$.

\begin{lemma}\label{L1}
For every $X\in\bP_m$,
$$
\|\log X\|\le\log(\|X\|+\|X^{-1}\|).
$$
Moreover, for every $r>0$ there exists a constant $c_r>0$ such that
$$
\|\log X\|_2\le c_r(\|X\|+\|X^{-1}\|)^r,\qquad X\in\bP_m.
$$
\end{lemma}

\begin{proof}
Since $\|X^{-1}\|^{-1}I\le X\le\|X\|I$, we have
$(-\log\|X^{-1}\|)I\le \log X\le(\log\|X\|)I$ so that
$$
\|\log X\|=\max\bigl\{\log\|X\|,\log\|X^{-1}\|\bigr\}\le\log(\|X\|+\|X^{-1}\|).
$$
Next, for any $r>0$, since $\lim_{x\to\infty}(\log x)/x^r=0$,
$b_r:=\sup_{x\ge1}(\log x)/x^r<\infty$. Noting that
$\|X\|+\|X^{-1}\|\ge2\sqrt{\|X\|\,\|X^{-1}\|}\ge2$, we have
\begin{align*}
\|\log X\|_2&\le\sqrt m\,\|\log X\|\le\sqrt m\,\log\bigl(\|X\|+\|X^{-1}\|\bigr) \\
&\le\sqrt m\,b_r\bigl(\|X\|+\|X^{-1}\|\bigr)^r,\qquad X\in\bP_m.
\end{align*}
\end{proof}

Now, for $\mu\in\cP(\bP_m)$ we consider the condition
\begin{equation}\label{F-4.1}
\int_{\bP_m}\bigl(\|X\|+\|X^{-1}\|\bigr)\,d\mu(X)<\infty.
\end{equation}

\begin{lemma}\label{L3}
If $\mu\in\cP(\bP_m)$ satisfies \eqref{F-4.1}, then
$\mu\in\cP^p(\bP_m)$ for every $p\in[1,\infty)$.
\end{lemma}

\begin{proof}
By Lemma \ref{L1} with $r=1/p$ we have
$$
\|\log X\|_2\le k_{1/p}\bigl(\|X\|+\|X^{-1}\|\bigr)^{1/p},\qquad
X\in\bP_m.
$$
Therefore,
$$
\int_{\bP_m}d^p(X,I)\,d\mu(X)=\int_{\bP_m}\|\log X\|_2^p\,d\mu(X)
\le
k_{1/p}^p\int_{\bP_m}\bigl(\|X\|+\|X^{-1}\|\bigr)\,d\mu(X)<\infty,
$$
implying $\mu\in\cP^p(\bP_m)$.
\end{proof}

When $\mu$ satisfies \eqref{F-4.1}, one can define the arithmetic and the
harmonic means of $\mu$ as
$$
\int_{\bP_m}X\,d\mu(X),\qquad\biggl(\int_{\bP_m}X^{-1}\,d\mu(X)\biggr)^{-1},
$$
respectively. By Lemma \ref{L3} one can also define the Cartan
barycenter $G(\mu)$.

The next lemma will be useful in the proof of our main result of this section.

\begin{lemma}\label{L4}
Assume that $\mu\in\cP(\bP_m)$ satisfies \eqref{F-4.1}. Then there
exist a sequence $\{\mu_n\}_{n=1}^{\infty}$ in $\cP_c(\bP_m)$ such that,
as $n\to\infty$,
$$
d_1^W(\mu_n,\mu)\longrightarrow0,
$$
and
$$
\int_{\bP_m}X\,d\mu_n(X)\longrightarrow\int_{\bP_m}X\,d\mu(X),\quad
\int_{\bP_m}X^{-1}\,d\mu_n(X)\longrightarrow\int_{\bP_m}X^{-1}\,d\mu(X).
$$
\end{lemma}

\begin{proof}
For each $n\in\bN$ let $\Sigma_n$ be as in the proof of Theorem \ref{T:mono} and
define $\mu_n\in\cP_c(\bP_m)$ as
$$
\mu_n:=\mu|_{\Sigma_n}+\mu(\bP_m\setminus\Sigma_n)\delta_{I}.
$$
We then have $d_1^W(\mu_n,\mu)\to0$  as in the proof of Theorem
\ref{T:mono}, since $\mu_n$ converges weakly to $\mu$ and
$$
\int_{\bP_m}\|\log X\|_2\,d\mu_n(X)
=\int_{\Sigma_n}\|\log X\|_2\,d\mu(X)\longrightarrow
\int_{\bP_m}\|\log X\|_2\,d\mu(X)
$$
as $n\to\infty$. On the other hand, by assumption \eqref{F-4.1} we have
$$
\int_{\bP_m}X\,d\mu_n(X)=\int_{\Sigma_n}X\,d\mu(X)+\mu(\bP_m\setminus\Sigma_n)I
\longrightarrow\int_{\bP_m}X\,d\mu(X),
$$
and similarly $\int_{\bP_m}X^{-1}\,d\mu_n(X)\to\int_{\bP_m}X^{-1}\,d\mu(X)$.
\end{proof}

The following {\it AGH $($arithmetic-geometric-harmonic$)$ mean
inequalities} were shown for $\mu\in\cP_0(\bP_m)$ in \cite[Theorem
2]{Ya2} and extended in \cite{KL} to the case of
$\mu\in\cP_c(\bP_m)$. We further extend it to the case of $\mu$
satisfying \eqref{F-4.1}.

\begin{proposition}[AGH inequalities]\label{P5}
If $\mu\in\cP(\bP_m)$ satisfies \eqref{F-4.1}, then
\begin{equation}\label{F-4.8}
\biggl(\int_{\bP_m}X^{-1}\,d\mu(X)\biggr)^{-1}\le G(\mu)\le
\int_{\bP_m}X\,d\mu(X).
\end{equation}
\end{proposition}

\begin{proof}
By Lemma \ref{L4} choose a sequence $\{\mu_n\}$ in $\cP_1(\bP_m)$
such that $d_1^W(\mu_n,\mu)\to0$ (hence $G(\mu_n)\to G(\mu)$ by
Theorem \ref{T:ft}) and
$$
\int_{\bP_m}X\,d\mu_n(X)\longrightarrow\int_{\bP_m}X\,d\mu(X),\quad
\int_{\bP_m}X^{-1}\,d\mu_n(X)\longrightarrow\int_{\bP_m}X^{-1}\,d\mu(X).
$$
Since inequalities \eqref{F-4.8} hold for $\mu_n$, the result follows
by taking the limit of \eqref{F-4.8} for $\mu_n$.
\end{proof}

\begin{remark}\rm
Let $X_{ij}$ and $(X^{-1})_{ij}$ denote the $(i,j)$-entries of
$X,X^{-1}$, respectively. Then it is clear that the functions
$X\in\bP_m\mapsto X_{ij},(X^{-1})_{ij}$ are integrable with respect
to $\mu$ for all $i,j=1,\dots,m$ if and only if condition
\eqref{F-4.1} holds. Hence \eqref{F-4.1} is the best possible assumption
for the AGH mean inequalities in Proposition \ref{P5} to make sense.
\end{remark}

\begin{lemma}\label{L7}
For every $\mu\in\cP(\bP_m)$ with \eqref{F-4.1},
$$
{1\over t}\log\int_{\bP_m}X^t\,d\mu(X) \longrightarrow\int_{\bP_m}\log X\,d\mu(X),
$$
or equivalently,
$$
\left(\int_{{\Bbb
P}_{m}}X^{t}d\mu(X)\right)^{1\over t}=\exp\int_{{\Bbb P}_{m}}\log X\,d\mu(X)
$$
as $t\to0$ with $|t|\le1$.
\end{lemma}

\begin{proof}
First, note that $\int_{\bP_m}\log X\,d\mu(X)$ exists by Lemma \ref{L1}. For any $X\in\bP_m$
we write
$$
X^t=e^{t\log X}=I+t\log X+R(t,X),
$$
where
$$
R(t;X):=\sum_{n=2}^\infty{t^n\over n!}(\log X)^n.
$$
Assuming $|t|\le1$ we have
$$
\|R(t,X)\|\le t^2\sum_{n=0}^\infty{1\over n!}\,\|\log X\|^n
=t^2e^{\|\log X\|} \le t^2\bigl(\|X\|+\|X^{-1}\|\bigr)
$$
by Lemma \ref{L1}. Therefore,
$$
\int_{\bP_m}\|R(t,X)\|\,d\mu(X)\le
t^2\int_{\bP_m}\bigl(\|X\|+\|X^{-1}\|\bigr)\,d\mu(X)
=O(t^2)\quad\mbox{as $t\to0$},
$$
so that we have
$$
\int_{\bP_m}X^t\,d\mu(X)=I+t\int_{\bP_m}\log X\,d\mu(X)+O(t^2).
$$
This implies that
$$
{1\over t}\log\int_{\bP_m}X^t\,d\mu(X) =\int_{\bP_m}\log
X\,d\mu(X)+O(t),
$$
and hence
$$
\lim_{t\to0}{1\over t}\log\int_{\bP_m}X^t\,d\mu(X)=\int_{\bP_m}\log
X\,d\mu(X).
$$
\end{proof}

Finally, for $\mu\in\cP(\bP_m)$ we consider the condition
\begin{equation}\label{F-4.9}
\int_{\bP_m}\bigl(\|X\|+\|X^{-1}\|\bigr)^r\,d\mu(X)<\infty
\end{equation}
for some $r>0$. It is obvious that if \eqref{F-4.9} holds for $r>0$,
then it also holds for any $r'\in(0,r]$. Moreover, for any $r>0$, condition
\eqref{F-4.9} is equivalent to
$$
\int_{\bP_m}\bigl(\|X^r\|+\|X^{-r}\|\bigr)\,d\mu(X)<\infty,
$$
so that both $\mu^r$ and $\mu^{-r}$ satisfy \eqref{F-4.1}.

Our main result of this section is the following:

\begin{theorem}[Lie-Trotter formula]\label{T8}
Let  $\mu\in\cP(\bP_m)$ satisfying \eqref{F-4.9} for some $r>0$. Then
we have
\begin{equation}\label{F-4.10}
\lim_{t\to0}G(\mu^t)^{1\over t}=\exp\int_{\bP_m}\log X\,d\mu(X).
\end{equation}
\end{theorem}

\begin{proof}
First, assume that $\mu$ satisfies \eqref{F-4.1}. For any
$t\in[-1,1]\setminus\{0\}$, by using Proposition \ref{P5} to $\mu^t$ we have
\begin{align*}
\biggl(\int_{\bP_m}X^{-t}\,d\mu(X)\biggr)^{-1}&=
\biggl(\int_{\bP_m}X^{-1}\,d\mu^t(X)\biggr)^{-1} \\
&\le G(\mu^t)\le\int_{\bP_m}X\,d\mu^t(X)=\int_{\bP_m}X^t\,d\mu(X).
\end{align*}
Since $\log x$ is operator monotone on $(0,\infty)$, the above
inequalities give
\begin{align*}
&-{1\over t}\log\int_{\bP_m}X^{-t}\,d\mu(X)
\le\log G(\mu^t)^{1\over t}\le{1\over t}\log\int_{\bP_m}X^t\,d\mu(X)
\quad\mbox{if $0<t\le1$}, \\
&-{1\over t}\log\int_{\bP_m}X^{-t}\,d\mu(X) \ge\log
G(\mu^t)^{1\over t}\ge{1\over t}\log\int_{\bP_m}X^t\,d\mu(X)
\quad\mbox{if $-1\le t<0$}.
\end{align*}
From Lemma \ref{L7} this implies that
\begin{equation}\label{F-4.11}
\lim_{t\to0}\log G(\mu^t)^{1\over t}=\int_{\bP_m}\log X\,d\mu(X).
\end{equation}

Next, assume that $\mu$ satisfies \eqref{F-4.9} for some $r>0$, that
is, $\mu^r$ satisfies \eqref{F-4.1}. The above case yields
$$
\lim_{t\to0}\log G\bigl((\mu^r)^t\bigr)^{1\over t}=\int_{\bP_m}\log X\,d\mu^r(X).
$$
Note that the left-hand side in the above is
$$
\lim_{t\to0}\log G(\mu^{rt})^{1\over t}=r\lim_{t\to0}\log G(\mu^t)^{1\over t},
$$
while the right-hand side is
$$
r\int_{\bP_m}\log X\,d\mu(X).
$$
Hence we have \eqref{F-4.11} again, which implies \eqref{F-4.10}.
\end{proof}

The next corollary extends \cite[Corollary 2]{BG} to the case of probability measures
satisfying \eqref{F-4.9}.

\begin{corollary}\label{C9}
Assume that $\mu\in\cP(\bP_m)$ satisfies \eqref{F-4.9} for an $r>0$
and $|||\cdot|||$ is any unitarily invariant norm. Then
\begin{itemize}
\item[(a)] For every $t>0$,
\begin{equation}\label{F-4.12}
\big|\big|\big|G(\mu^{-t})^{-{1\over t}}\big|\big|\big|
=\big|\big|\big|G(\mu^t)^{1\over t}\big|\big|\big|
\le\bigg|\bigg|\bigg|\exp\int_{\bP_m}\log X\,d\mu(X)\bigg|\bigg|\bigg|,
\end{equation}
and $\big|\big|\big|G(\mu^t)^{1\over t}\big|\big|\big|$
increases to $\big|\big|\big|\exp\int_{\bP_m}\log X\,d\mu(X)\big|\big|\big|$
as $t\searrow0$.
\item[(b)] If $0<t\le r$, then
\begin{align}
&\bigg|\bigg|\bigg|\left(\int_{{\Bbb P}_{m}}X^{-t}
\,d\mu(X)\right)^{-{1\over t}}\bigg|\bigg|\bigg|
\le\big|\big|\big|G(\mu^t)^{1\over t}\big|\big|\big| \nonumber\\
&\quad\leq\bigg|\bigg|\bigg|\exp\int_{\bP_m}\log
X\,d\mu(X)\bigg|\bigg|\bigg|\leq
\bigg|\bigg|\bigg|\left(\int_{{\Bbb P}_{m}}X^{t}
\,d\mu(X)\right)^{1\over t}\bigg|\bigg|\bigg|. \label{F-4.13}
\end{align}
Furthermore,
$\big|\big|\big|\bigl(\int_{{\Bbb P}_{m}}X^{t}\,d\mu(X)\bigr)^{1\over t}\big|\big|\big|$
decreases to $\big|\big|\big|\exp\int_{\bP_m}\log X\,d\mu(X)\big|\big|\big|$ and
$\big|\big|\big|\bigl(\int_{{\Bbb P}_{m}}X^{-t}\,d\mu(X)\bigr)^{-{1\over t}}\big|\big|\big|$
increases to $\big|\big|\big|\exp\int_{\bP_m}\log X\,d\mu(X)\big|\big|\big|$ as
$r\ge t\searrow0$.
\end{itemize}
\end{corollary}

\begin{proof}
When $\mu\in\cP^1(\bP_m)$ (without condition \eqref{F-4.9}), from the invariance
$G(\mu^{-1})=G(\mu)^{-1}$ as immediately seen from Theorem \ref{T:kare},
we find that $G(\mu^{-t})^{-{1\over t}}=G(\mu^t)^{1\over t}$,
implying the equality in \eqref{F-4.12}. It follows from \eqref{E:LL} that
$\big|\big|\big|G(\mu^t)^{1\over t}\big|\big|\big|$ is increasing as $t\searrow0$.
In the rest, assume \eqref{F-4.9} for an $r>0$.

(a)\enspace
The inequality in \eqref{F-4.12} is immediately seen from Theorem \ref{T8} together with
$\big|\big|\big|G(\mu^t)^{1\over t}\big|\big|\big|$ being increasing noted above.

(b)\enspace
Assume that $0<t'<t\le r$ and prove that
\begin{align}
\int_{\bP_m}X^{t'}\,d\mu(X)
&\le\biggl(\int_{\bP_m}X^t\,d\mu(X)\biggr)^{t'\over t}, \label{F-4.14}\\
\biggl(\int_{\bP_m}X^{-t'}\,d\mu(X)\biggr)^{-1}
&\ge\biggl(\int_{\bP_m}X^{-t}\,d\mu(X)\biggr)^{-{t'\over t}}. \label{F-4.15}
\end{align}
For each $n\in\bN$ let $\Sigma_n$ be as in the proof of Lemma \ref{L4}.
Since $X^t$ and $X^{t'}$ are uniformly continuous on the compact set $\Sigma_n$, one can
choose a sequence of simple functions $\sum_{j=1}^{k_\ell}A_{\ell,j}1_{\cQ_{\ell,j}}$,
$\ell\in\bN$, with $A_{\ell,j}\in\Sigma_n$ and Borel partitions
$\{\cQ_{\ell,j}\}_{j=1}^{k_\ell}$ of $\Sigma_n$ such that, as $\ell\to\infty$,
$$
\sum_{j=1}^{k_\ell}A_{\ell,j}^t\mu(\cQ_{\ell,j})\longrightarrow
\int_{\Sigma_n}X^t\,\mu(X),\qquad
\sum_{j=1}^{k_\ell}A_{\ell,j}^{t'}\mu(\cQ_{l,j})\longrightarrow
\int_{\Sigma_n}X^{t'}\,\mu(X).
$$
Due to the operator concavity of $x^{t'/t}$ on $(0,\infty)$, we have
$$
\sum_{j=1}^{k_\ell}\mu(\cQ_{\ell,j})A_{\ell,j}^{t'}+\mu(\bP_m\setminus\Sigma_n)I
\le\Biggl(\sum_{j=1}^{k_\ell}\mu(\cQ_{\ell,j})A_{\ell,j}^t
+\mu(\bP_m\setminus\Sigma_n)I\Biggr)^{t'\over t}.
$$
Letting $l\to\infty$ gives
$$
\int_{\Sigma_n}X^{t'}\,d\mu(X)+\mu(\bP_m\setminus\Sigma_n)I \le
\biggl(\int_{\Sigma_n}X^t\,d\mu(X)+\mu(\bP_m\setminus\Sigma_n)I\biggr)^{t'\over t}.
$$
Since $\|X^t\|$ and $\|X^{t'}\|$ are integrable with respect to $\mu$, \eqref{F-4.14}
follows by taking the limit of the above inequality as $n\to\infty$. Then, \eqref{F-4.15}
also follows by replacing $\mu$ with $\mu^{-1}$ in \eqref{F-4.14}. Now, similarly to
the proof of \cite[Theorem~1]{BG} we see that for $1\le j\le m$, as $r\ge t\searrow0$,
the $j$th eigenvalue of $\bigl(\int_{\bP_m}X^t\,d\mu(X)\bigr)^{1\over t}$ is decreasing
and that of $\bigl(\int_{\bP_m}X^{-t}\,d\mu(X)\bigr)^{-{1\over t}}$ is increasing.

Furthermore, by applying Lemma \ref{L7} to $\mu^r$ we have
$$
\biggl(\int_{\bP_m}X^t\,d\mu^r(X)\biggr)^{1\over t}\ \longrightarrow
\ \exp\int_{\bP_m}\log X\,d\mu^r(X)\quad\mbox{as $t\to0$ with $|t|\le1$},
$$
which is rephrased as
$$
\biggl(\int_{\bP_m}X^t\,d\mu(X)\biggr)^{1\over t}\ \longrightarrow
\ \exp\int_{\bP_m}\log X\,d\mu(X)\quad\mbox{as $t\to0$ with $|t|\le r$}.
$$
Hence, as $r\ge t\searrow0$,
$\big|\big|\big|\bigl(\int_{{\Bbb P}_{m}}X^{t}\,d\mu(X)\bigr)^{1\over t}\big|\big|\big|$
decreases to $\big|\big|\big|\exp\int_{\bP_m}\log X\,d\mu(X)\big|\big|\big|$ while
$\big|\big|\big|\bigl(\int_{{\Bbb P}_{m}}X^{-t}\,d\mu(X)\bigr)^{-{1\over t}}\big|\big|\big|$
increases to the same. In view of (a) it remains to show the first inequality in
\eqref{F-4.13}. But this is immediately seen by applying \eqref{F-4.8} to $\mu^t$ for
$0<t\le r$.
\end{proof}

\begin{remark}\rm
The following example shows that condition \eqref{F-4.9} is not satisfied for any $r>0$
even if we have $\mu\in\cP^p(\bP_m)$ for all $p>0$. For instance, choose $X_n\in\bP_m$
such that $X_n\ge I$ and $\|X_n\|=n^n$, and define
$$
\mu:=\sum_{n=1}^\infty{1\over2^n}\,\delta_{X_n}.
$$
Then, for any $r>0$,
$$
\int_{\bP_m}\|X\|^r\,d\mu(X)=\sum_{n=1}^\infty{(n^r)^n\over2^n}=\infty,
$$
while
$$
\int_{\bP_m}\|\log X\|_2^p\,d\mu(X)
\le\sum_{n=1}^\infty{\bigl(m\log^2\|X_n\|\bigr)^{p/2}\over2^n}
=m^{p/2}\sum_{n=1}^\infty{(n\log n)^p\over2^n}<\infty
$$
for all $p>0$.
\end{remark}

\begin{problem}\rm
Do Theorem \ref{T8} and part (a) of Corollary \ref{C9} hold for
general $\mu\in\cP^1(\bP_m)$ without assumption \eqref{F-4.9}? In
part (b) of Corollary \ref{C9}, we cannot define $\int_{\bP_m}X^{\pm
t}\,d\mu(X)$ for general $\mu\in\cP^1$, while part (a) makes sense
for general $\mu\in\cP^1$.
\end{problem}

\subsection*{Acknowledgments}
 The
authors thank Hiroyuki Osaka and Takeaki Yamazaki  for inviting the
workshop on Quantum Information Theory and Related Topics 2016 in
Ritsumeikan University  where this work was initiated. The work of
F.~Hiai was supported by Grant-in-Aid for Scientific Research
(C)26400103. The work of Y.~Lim was supported by the National
Research Foundation of Korea (NRF) grant funded by the Korea
government(MEST) No.2015R1A3A2031159 and 2016R1A5A1008055.

\end{document}